\documentclass{article}
\usepackage{graphicx,amssymb}
\usepackage{lineno}
\usepackage{mathptmx}      
\usepackage{color}
\usepackage[dvipsnames]{xcolor}

\usepackage{amsmath,amsthm,tikz}
\usepackage[a4paper, total={5.4in, 9.6in}]{geometry}
\usetikzlibrary{calc,decorations.pathreplacing,fadings}
\newtheorem{theorem}{Theorem}
\newtheorem{lemma}[theorem]{Lemma}
\newtheorem{corollary}[theorem]{Corollary}

\newcommand{\rar}{\rightarrow}

\begin{document}

\title{Conditional measure and the violation of Van Lambalgen's theorem for Martin-L\"of randomness}

\date{}
\author{Bruno Bauwens
\thanks{
      National Research University Higher School of Economics (HSE),\newline
      Faculty of Computer Science,\newline
      Kochnovskiy Proezd 3, Moscow, 125319 Russia.\newline
      BBauwens\color{white}{x}\color{black}at\color{white}{z}\color{black}hse\color{white}{u}\color{black}dot\color{white}{u}\color{black}ru
    \smallskip
    \newline This research was done on leave from Ghent University to the Higher School of Economics.  
    The work was done partially while the author was visiting the Institute for Mathematical Sciences, 
    National University of Singapore in 2014.  The visit was supported by the Institute.  
    \smallskip
    \newline 
    Hayato Takahashi and I visited Alexander (Sasha) Shen at Universit\'e de Montpellier 2 in France in
    November 2012 (both visits where supported by the project NAFIT  ANR-08-EMER-008-01).  
    I am very grateful for the intense discussion on the proof of
    Theorem~\ref{th:TakahashiVanLambalgen} and for raising the question answered in this note.
    Understanding this proof was important to prove the main result here.  
    Several simplifications of the proof here were made also by Alexander Shen.
    \smallskip
    \newline 
    I thank Jason Rute for useful discussions and the reviewers for quick feedback and useful comments.
    \bigskip
    }
}

\maketitle
\begin{abstract}
  Van Lambalgen's theorem states that a pair $(\alpha,\beta)$ of bit sequences is Martin-L\"of random
  if and only if $\alpha$ is Martin-L\"of random and $\beta$ is Martin-L\"of random
  relative to $\alpha$. 
  In [Information and Computation 209.2 (2011): 183-197, Theorem 3.3], Hayato Takahashi generalized 
  van Lambalgen's theorem for computable measures $P$ on a product of two Cantor spaces; he showed that 
  the equivalence holds for each $\beta$ for which the conditional probability $P(\cdot | \beta)$ is computable. 
  He asked whether this computability condition is necessary.
  We give a positive answer by providing a computable measure for which van Lambalgen's theorem fails.
  We also present a simple construction of a computable measure for which conditional measure is not computable.  
  Such measures were first constructed by N. Ackerman, C. Freer and D. Roy 
  in [Proceedings of the 26th Annual IEEE Symposium on Logic in Computer Science (LICS), pp. 107-116. IEEE (2011)].
\end{abstract}

Michiel van Lambalgen characterized Martin-L\"of randomness of a pair of bit sequences:

\begin{theorem}[van Lambalgen~\cite{VanLambalgen}]\label{th:VanLambalgen}
  The following are equivalent for a pair $(\alpha,\beta)$ of sequences:
  \begin{itemize}
    \item $(\alpha,\beta)$ is Martin-L\"of random,
    \item $\beta$ is Martin-L\"of random and $\alpha$ is Martin-L\"of random relative to $\beta$.
  \end{itemize}
\end{theorem}
\noindent
One can replace uniform (Lebesgue) measure in the definition of Martin-L\"of randomness
by any measure~$P$. We call sequences that are random 
in this sense {\em$P$-random}.\footnote{\label{foot}
      There exist two types of Martin-L\"of tests relative to a non-computable measure~$P$~\cite{HanssenHippocratic}:
      \begin{itemize}
	\item 
	  A {\em uniform} $P$-Martin-L\"of test 
	  is a $P$-Martin-L\"of test that is effectively open relative to each oracle that computes~$P$.
	\item 
	  A  {\em Hippocratic} or {\em blind} $P$-Martin-L\"of test is a Martin-L\"of test that is effectively open 
	  without any oracle. 
      \end{itemize}
      If $P$ is computable, then both types of tests define the same set of random sequences.
      Otherwise, the second type of tests defines a weaker notion of randomness, which
      we use here. 
      Although the definition looks easier, it might not be the most
      natural definition of randomness relative to a non-computable
      measure~$P$.
  }
There exist two definitions of Martin-L\"of randomness 
for a pair of sequences. The first states that $(\alpha,\beta)$ is random if the join 
$\alpha_{^1}\beta_{^1}\alpha_{^2}\beta_{^2}\dots$ is random. The second definition uses 
the two dimensional variant of a Martin-L\"of test, which is given by a family of
uniformly effectively open sets $U_n \subseteq 2^{\mathbb N} \times 2^{\mathbb N}$ 
such that the uniform measure of $U_n$ is at most $2^{-n}$. Both approaches are 
equivalent.

To generalize van Lambalgen's theorem for computable measures $P$, the first 
approach seems not suitable. Why join two sequences in this specific way? What
does it mean? Also, the most direct approach of replacing Martin-L\"of randomness with
$P$-randomness will make the theorem wrong for trivial reasons: There exist a
computable $P$ and a pair of sequences
$(\alpha,\beta)$ such that $\alpha_{^1}\beta_{^1}\alpha_{^2}\beta_{^2}\dots$ is
$P$-random, while $\alpha$ is not $P$-random.  Indeed, let $P$ be the measure
that concentrates all its mass on the single point $010101\dots$, i.e.,
$P(\{0101\dots\}) = 1$ and $P(S) = 0$ if $0101\dots \not\in S$.  The
sequence $0101\dots$ is $P$-random, but $00\dots$ is not random.

To use the two-dimensional approach, we need to decompose the bivariate measure $P$ 
into two univariate measures.
It is natural to use the marginal and  the conditional 
measure for~$P$. In fact, such decompositions are omnipresent in probability theory,
and it nicely fits the statement of van Lambalgen's theorem, which uses in
the second criterion a conditionally and an unconditionally random sequence.

\bigskip
We now define conditional measure.
Let $2^\mathbb{N}$ denote Cantor space. 
For any string $x$, let $[x]$ be the (basic open) set containing all extensions of~$x$. 
We say that a measure $P$ on $2^\mathbb{N}$ is computable if the function that maps 
each string $x$ to $P([x])$ is computable as a real-valued function. 
Similar for measures $P$ on $2^\mathbb{N} \times 2^\mathbb{N}$.
Following Takahashi~\cite{TakahashiDefinition}, 
we define for each measure $P$ on $2^\mathbb{N} \times 2^\mathbb{N}$, for each $\beta \in 2^\mathbb{N}$,
and for each measurable set $S \subseteq 2^\mathbb{N}$:
\[
P_C(S|\beta) = \lim_{n \rar \infty}
\frac{P(S \times [\beta_1\dots\beta_n])}{P(2^\mathbb{N} \times [\beta_1\dots\beta_n])}.
\]
Let the marginal distribution be $P_M(S) = P(2^\mathbb{N} \times S)$.  

\smallskip
Remark:
The definition of a conditional measure is usually given using the Radon-Nikodym theorem. 
In fact, this theorem defines a set of conditional measures, and each pair of such measures 
coincides on a set $\beta$ of $P_M$-measure one.
Using the Lebesgue differentiation theorem it can be shown that these conditional 
measures also coincide with $P_C(\cdot|\beta)$ for $P_M$-almost all $\beta$. 
We refer to the appendix for more details.

\smallskip
This specific conditional measure is especially suitable to generalize van Lambalgen's
theorem because of the following result 
(see also~\cite[Lemma 10]{JasonRandFromRand}): 
\begin{theorem}[\cite{TakahashiDefinition} Takahashi]\label{th:condDefined}
If $\beta$ is $P_M$-random, then $P_C(\cdot|\beta)$ is defined 
and is a measure.
\end{theorem}
In~\cite[Theorem 29, p14]{noncomputableConditionalMeasure} 
it is shown that for computable $P$, the measure $P_C$ might not be computable. 
The measure that satisfies the conditions of our main result satisfies a similar property:
\begin{corollary}[of the proof of Theorem~\ref{th:VanLambalgenFails}]\label{cor:noncomputableConditional}
  There exists a computable measure $P$ on $2^{\mathbb{N}} \times 2^{\mathbb{N}}$ such that 
  the set of $\beta$ for which $P_C(\cdot|\beta)$ is not computable relative to~$\beta$, has nonzero $P_M$-measure.
\end{corollary}
\newcounter{cor}
\setcounter{cor}{\value{theorem}}
The corollary is proven after Theorem~\ref{th:VanLambalgenFails}. Similar examples 
of such measures were invented by Jason Rute~\cite{JasonRuteNoncomputable}.
In the example from~\cite{noncomputableConditionalMeasure}, 
definitions of computability of functions and measures from computable analysis are used. 
They can be used on general spaces but are rather difficult to formulate.
Functions that are not computable in this sense include all functions with a discontinuity. 
Therefore, the example in~\cite{noncomputableConditionalMeasure} is made in such a way that
$P_C(S|\beta)$ is continuous in $\beta$ for all measurable sets $S$.
Using the same idea as in~\cite{noncomputableConditionalMeasure}, we present a simplified 
construction of such a measure in the proof of Theorem~\ref{th:continuousNoncomputable} below.
  This proof does not rely on other parts of this note.

\medskip
Hayato Takahashi generalized van Lambalgen's theorem as follows: 
\begin{theorem}[Takahashi~\cite{TakahashiProduct,TakahashiGeneralization}]\label{th:TakahashiVanLambalgen}
  For any computable bivariate measure $P$ and any $\beta$ such that $P_C(\cdot|\beta)$ is
  computable relatively to $\beta$, the following are equivalent:
  \begin{itemize}
    \item $(\alpha,\beta)$ is $P$-random,
    \item $\beta$ is $P_M$-random and $\alpha$ is $P_C(\cdot|\beta)$-random relative to~$\beta$.
  \end{itemize}
\end{theorem}
For an alternative exposition of the proof and for related results,
I refer to the upcoming article~\cite{compVanLambalgenOverview}.
One might ask whether the theorem only holds for $\beta$ for which
$P_C(\cdot|\beta)$ is computable relative to~$\beta$? 
Our main result shows that we can not drop this assumption, hence, 
van Lambalgen's theorem fails for some computable measure. 
We emphasize that for a non-computable measure $Q$, our definition of
$Q$-randomness corresponds to what is usually called {\em blind} or  {\em
Hippocratic} randomness in the literature (see footnote~\ref{foot} for more
details).
\begin{theorem}\label{th:VanLambalgenFails}
  There exists a bivariate computable measure $P$ on $2^{\mathbb N} \times 2^{\mathbb N}$ 
  and a pair of sequences $(\alpha,\beta)$ such that the pair is $P$-random 
  and $\alpha$ is not $P_C(\cdot|\beta)$-random (thus even without oracle $\beta$).
\end{theorem}
Note that if $(\alpha,\beta)$ is $P$-Martin-L\"of random, then $\beta$ is $P_M$-Martin-L\"of random, 
and by Theorem~\ref{th:condDefined} the measure $P_C(\cdot|\beta)$ exists.

\medskip
\noindent
 {\em Definitions}\, 
Let $\mu$ be the uniform measure, i.e., $\mu([x]) = 2^{-|x|}$ for any string~$x$. 
We also use $\mu$ for the product of two uniform measures over $2^{\mathbb N} \times 2^{\mathbb N}$.
Real numbers in $[0,1]$ that are not binary rational, are interpreted as elements of $2^\mathbb{N}$. 
For binary rational numbers $\alpha$ and
$\beta$, we associate $[\alpha, \beta]$ with the corresponding basic open set in Cantor space 
(thus only containing the binary representation of $\alpha$ with a tail of zeros, and a tail of ones for $\beta$).

\newcommand{\drawlines}{
  \draw (0,\a)   node[anchor=north,gray]  {\footnotesize{$\alpha_1$}}  -- +(4,0);
  \draw (0,\aa)  node[anchor=north,gray] {\footnotesize{$\alpha_2$}}  -- +(4,0);
  \draw (0,\aaa) node[anchor=north,gray] {\footnotesize{$\alpha_3$}}  -- +(4,0);
  \draw (0,\afour)   -- +(4,0);
  \draw       (0,\alim)  node[anchor=north] {$\alpha$}    -- +(4,0);

  \draw[line width=2.5pt] (0.00,0) -- +(0,\a);

  \foreach \x  in {0.00,2.00}
    {\draw[line width=2pt] (\x,\a) -- (\x,\aa);}
  \foreach \x  in {0.0,1,...,3} 
    {\draw[line width=1.5pt] (\x,\aa) -- (\x,\aaa);}
  \foreach \x  in {0.00,0.50,...,3.50} {
    \draw[line width=1pt] (\x,\aaa) -- (\x,\afour);
    }
  \foreach \x  in {0,0.25,...,3.75} 
    {\draw[gray] (\x,\afour) -- (\x,\alim);}
  \draw rectangle (4,4);
  }

\begin{figure}
  \def\a{1.2}
  \def\aa{1.9}
  \def\aaa{2.3}
  \def\afour{2.55}
  \def\alim{2.7}

  \centering
  \begin{tikzpicture}[yscale=1.2,xscale=-1.2,rotate=90]
    \filldraw[fill=gray!20] (0,\alim) rectangle (4,4);
    \drawlines
  \end{tikzpicture}
  \quad\quad
  \begin{tikzpicture}[yscale=1.2,xscale=-1.2,rotate=90]
    \filldraw[fill=gray!20] (0,\alim) rectangle (4,4);

    \fill[dashed, fill=gray!35] (0,0.5) rectangle +(3.98,0.6);
    \draw[dashed, draw=gray] (0,0.5) rectangle +(1,0.6);
    \draw[decorate, decoration={brace,amplitude=3pt},xshift=-1pt,gray] (0,1.1) -- node[anchor=north,yshift=-1pt] {$I$} +(0,-0.6);

    \draw[decorate, decoration={brace,amplitude=3pt},yshift=-1pt,gray] (0.05,0.5) -- node[anchor=east] (tt) {\small{$[x]=[00]$}} +(0.95,0);

    \draw[dashed, draw=gray] (2,0.5) rectangle +(1,0.6);
    \draw[decorate, decoration={brace,amplitude=3pt},yshift=-1pt,gray] (2.05,0.5) -- node[anchor=east] (tt) {\small{$[x]=[10]$}} +(0.95,0);

    \fill[dashed, fill=gray!35] (0,1.4) rectangle +(1.96,0.4);
    \draw[dashed, draw=gray] (0,1.4) rectangle +(1,0.4);
    \draw[decorate, decoration={brace,amplitude=3pt},xshift=-1pt,gray] (0,1.8) -- node[anchor=north,yshift=-1pt] {$I$} +(0,-0.4);

    \filldraw[dashed, draw=blue, fill=blue!25] (1.98,2.1) rectangle (2.98,3.2);
    \node[anchor=north,blue] at (1.98,2.1) {$r$};
    \node[anchor=north,blue] at (1.98,3.2) {$s$};

    \drawlines
  \end{tikzpicture}
  \caption{Left: The measure $P$. Thick black lines represent concentrated measure.
  Right: Some values for $P(I \times [x])$. From left to right: 
     For $I \subseteq [0, \alpha_1]$, 
     $P(I \times [00]) = \mu(I)$ and $P(I \times [10]) = 0$.
     For $I \subseteq [\alpha_1, \alpha_2]$, we have $P(I \times [00]) = \mu(I)/2$.
     For $\alpha_{|x|} \le r < s \le 1$ we have $P([r,s] \times [x]) = \mu([r,s] \times [x])$.}
  \label{fig:constructionP}
\end{figure}
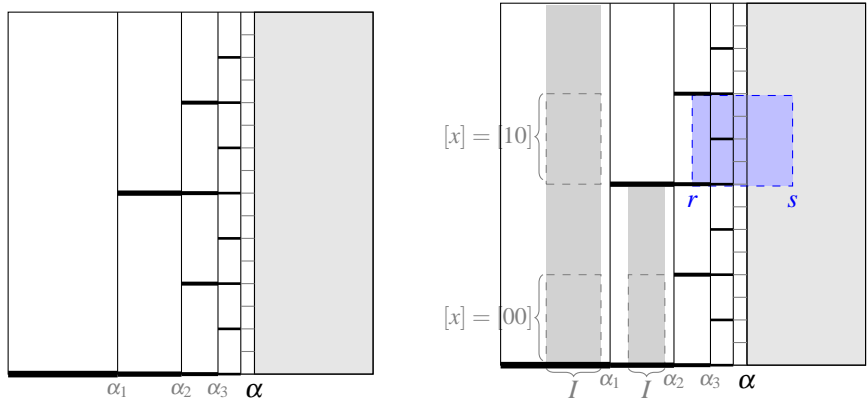

\begin{proof}
  There exists an increasing computable sequence $\alpha_1$, $\alpha_2$, \dots of binary rational numbers that
  converges to a Martin-L\"of random real $\alpha$. (See e.g.~\cite[Theorem 4.3]{Chaitin75}.) 
  To construct the bivariate measure $P$,
  modify the uniform measure on $2^\mathbb{N} \times {2^\mathbb{N}}$ as illustrated in figure~\ref{fig:constructionP}
  left: concentrate all measure in the vertical strip $[0,\alpha_1] \times 2^\mathbb{N}$ 
  uniformly in its lowest horizontal interval, i.e., in $[0,\alpha_1] \times 000\dots$;  
  concentrate the measure in the intervals $[\alpha_1,\alpha_2] \times [0]$ and $[\alpha_1,\alpha_2] \times [1]$
  uniformly in their lowest positions, 
  i.e., in $[\alpha_1, \alpha_2] \times 000\dots$ and $[\alpha_1, \alpha_2] \times1000\dots$; and so on.\footnote{
    The construction has some similarities with the measure 
    constructed in the proof of Proposition 6.3 in~\cite{BienvenuPorter}: 
    the measure has also singularities that approach a left computable real. 
    However, I believe there is no deeper correspondence between this measure and the measure 
    constructed here.
    }

  Before presenting the formal definition, let us
  illustrate the construction of $P$. Consider an interval $I \subseteq [0, \alpha_1]$. 
  We have $P(I \times [x]) = 0$ if $x$ contains at least one $1$, and $P(I \times [x]) = \mu(I)$ otherwise,
  see figure~\ref{fig:constructionP} right. 
  For $I \subseteq [\alpha_1, \alpha_2]$ we have $P(I \times [0x]) = P(I \times
  [1x]) = 0$ if $x$ contains at least one $1$ and $\mu(I)/2$ otherwise.  
  
  \medskip
  
  We define the measure more formally for every basic open set $I \times [y]
  \subseteq 2^\omega \times 2^\omega$. 
  We consider several cases:
  \begin{itemize}
    \item If $I \subseteq [\alpha, 1]$, then $P(I \times [y]) = \mu(I \times [y])$.
    \item If $I \subseteq [\alpha_n, \alpha_{n+1}]$ and $|y| \ge n$, 
      then let $y = wx$ where $w$ represents the first $n$ bits of $y$. If:
      \begin{itemize}
	\item $x$ contains at least one $1$, then $P(I \times [wx]) = 0$,
	\item otherwise, i.e. if $x$ is empty or contains only zeros, $P(I \times [wx]) = \mu(I \times [w])$.
      \end{itemize}
    \item Otherwise, we partition the basic open set in (countably many) other basic open sets that satisfy one of the 
      conditions above. The measure is the sum of the measures of all sets in the partition. 
  \end{itemize}
  
 \noindent 
  Note that for any string $x$, the $P$-measure of
  $[\alpha_{|x|}, \alpha] \times [x]$ equals the uniform measure, i.e., 
  \[
    P([\alpha_{|x|}, \alpha] \times [x]  ) = \mu([\alpha_{|x|}, \alpha] \times [x] ). 
  \]
  The same holds for any set $[r, s] \times [x]$ with $\alpha_{|x|} \le r < s \le 1$:
  $
    P([r,s] \times [x]  ) = \mu([r,s] \times [x] ) 
  $,
  see figure~\ref{fig:constructionP} right.
  $P$ is computable, because $\alpha_{|x|}$ is computable from $x$, and
  the vertical line at $\alpha_{|x|}$ splits any interval $I \times [x]$ in at most two parts, see figure~\ref{fig:constructionTest}; 
  for each part the $P$-measure is easily calculated: 
  the measure of the part at the right of $\alpha_{|x|}$ equals its uniform measure, and the measure in the 
  left part can be partitioned into finitely many pieces which each satisfies one of the cases 
  in the definition of~$P$.

  \medskip
  We choose $\beta$, such that $(\alpha,\beta)$ is Martin-L\"of random relative to the uniform measure. 
  By the original version of van Lambalgen's theorem, it suffices to choose $\beta$ to be 
  random relative to $\alpha$. 
  Clearly, $\beta$ contains infinitely many ones. 

  \medskip
  We show that the pair $(\alpha,\beta)$ is also $P$-random.
  Let $(V_n)_{n \in \mathbb{N}}$ be a Martin-L\"of test relative to~$P$. It suffices to
  convert this test to a Martin-L\"of test $(U_n)_{n \in \mathbb{N}}$ relative to 
  the \textbf{u}niform measure 
  such that $V_n$ and $U_n$ have the same intersection with the vertical line at position $\alpha$. 
  More precisely, it suffices for each $U_n$ to be uniformly effectively open such that:
  \begin{itemize}
    \item $V_n \cap \left( \{\alpha\} \times 2^\mathbb{N}  \right) 
      = U_n \cap \left( \{\alpha\} \times 2^\mathbb{N} \right)$,
    \item $\mu(U_n) \le P(V_n)$.
  \end{itemize}
  (Indeed, this implies that if $(\alpha,\beta)$ was not~$P$-random, then it is also not 
  random relative to the uniform measure and this would contradict the construction.)
  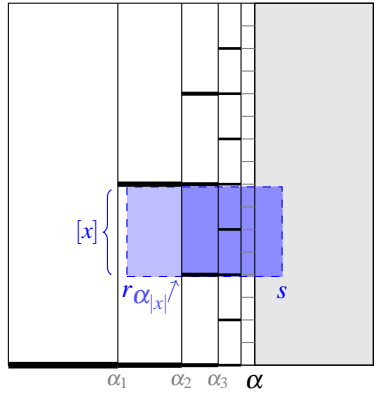
\begin{figure}
      \def\a{1.2}
      \def\aa{1.9}
      \def\aaa{2.3}
      \def\afour{2.55}
      \def\alim{2.7}
    \centering
  \begin{tikzpicture}[yscale=1.2,xscale=-1.2,rotate=90]
    \filldraw[fill=gray!20] (0,\alim) rectangle (4,4);

      \filldraw[dashed, blue!25] (0.98,1.3) rectangle (1.97,3);
      \filldraw[dashed, blue!45] (0.98,\aa) rectangle (1.97,3);
      \draw[dashed, blue] (0.98,1.3) rectangle (1.97,3);
      
      \node[anchor=north,blue] at (0.98,1.3) {$r$};
      \node[anchor=north,blue] at (0.98,3) {$s$};
      \draw[->,blue!70] (0.7,1.75) node[anchor=east,blue!70,xshift=1.5mm] {$\alpha_{|x|}$} -- (0.94,1.84) ;
      \draw[decorate, decoration={brace,amplitude=3pt},yshift=-2pt,blue] (1.00,1.2) -- node[anchor=east] (tt) {\small{$[x]\,$}} +(0.94,0);

    \drawlines
  \end{tikzpicture}
%
%
%
%
%
%
%

    \caption{The Martin-L\"of test $U_n$ for $\mu$ is obtained by trimming 
    each basic set enumerated into a set $V_n$.}
    \label{fig:constructionTest}
  \end{figure}
  Construction of $U_n$: Each time an interval $[r,s] \times [x]$ is enumerated into a set~$V_n$,
  enumerate its right part starting from $\alpha_{|x|}$ into $U_n$, i.e., 
  $[\max\{r,\alpha_{|x|}\},s] \times [x]$ if $s > \alpha_{|x|}$ and nothing otherwise, see
  figure~\ref{fig:constructionTest}.
  Note that $V_n$ and $U_n$ have the same intersection with the line at $\alpha$ because 
  enumerated intervals are only modified at the left of $\alpha_{|x|}<\alpha$.
  The sets $U_n$ are uniformly effectively open. Hence, the first condition is satisfied. 
  Finally, observe that $\mu(U_n) \le P(V_n)$:
  for each enumerated interval $[r,s] \times [x]$, nothing is changed unless $r < \alpha_{|x|}$, 
  and in this case we have
  $\mu([\alpha_{|x|},s] \times [x]) = P([\alpha_{|x|},s] \times [x]) \le
  P([r,s] \times [x]\,)$.
  Because $U_n$ and $V_n$ are the union of corresponding rectangles, the second condition is also 
  satisfied.

  \medskip
  Because $(\alpha,\beta)$ is $P$-random, $\beta$ is $P_M$-random and $P_C(\cdot|\beta)$ is defined.
  It remains to show that $\alpha$ is not $P_C(\cdot|\beta)$ random.
  We determine $P_C(\cdot|\beta)$.
  Observe that if $b$ is the empty string or a string that ends with a one, the measure
  \[
  \frac{P(\cdot \times [b])}{P_M([b])}
  \]
  is the uniform measure with support $[\alpha_{|b|},1]$.
  Because $\beta$ is $\mu$-random by construction, it contains infinitely many ones. 
  Hence, the limit in the definition of $P_C(\cdot|\beta)$ 
  (which exists), must be the uniform measure with support $[\alpha,1]$.
  This implies that the point $\alpha$ is not $P_C(\cdot|\beta)$-random: 
  the open sets $U_n =\; ]0, \alpha+2^{-n}[$ contain $\alpha$ for all $n$, are uniformly
  effectively open, and have $P(\cdot|\beta)$-measure $O(2^{-n})$. 
\end{proof}

\medskip
In the proof of Corollary~\ref{cor:noncomputableConditional}, we use the following observation: 
\begin{lemma}[De Leeuw, Moore, Shannon and Shapiro~\cite{DeLeeuwEtAl1955}]\label{lem:probComp}
  Let $Q$ be a computable measure on $2^\mathbb{N}$ and let $\alpha \in 2^\mathbb N$. If there exists a set of 
  positive $Q$-measure of sequences that compute $\alpha$, then $\alpha$ is computable.
\end{lemma}
\begin{proof}
  Because there exist only countably many machines, 
  there exists a unique machine that computes $\alpha$ from a set of sequences with
  positive $Q$-measure. Let $c>0$ be a lower bound for this $Q$-measure.
  We can enumerate a binary tree containing all strings $x$ that can be computed on this machine 
  from a set of oracles that has $Q$-measure at least $c$.
  This tree contains at most $1/c$ branches and each such branch is computable. 
\end{proof}

\noindent
We first repeat the corollary.
\newcounter{oldcounter}
\setcounter{oldcounter}{\value{theorem}}
\setcounter{theorem}{\value{cor}}
\addtocounter{theorem}{-1}
\begin{corollary}
  There exists a computable measure $P$ on $2^{\mathbb{N}} \times 2^{\mathbb{N}}$ such that 
  the set of $\beta$ for which $P_C(\cdot|\beta)$ is not computable relative to~$\beta$, has nonzero $P_M$-measure.
\end{corollary}
\setcounter{theorem}{\value{oldcounter}}
\begin{proof}
 Let $\alpha$ and $P$ be as constructed above. 
 Recall that $\alpha$ is random and therefore not computable. $P_M$ is computable. 
 The binary rational sequences have $P_M$-measure~$\alpha<1$, because
 $P$ concentrates all measure at the left of $\alpha$ on the binary rational sequences, 
 and at the right of $\alpha$, these sequences have $P_M$-measure zero.
 Let $B$ be the set of $P_M$-random sequences $\beta$ that are not binary rational. This set has 
 $P_M$-measure $1-\alpha>0$. 
 It remains to show that for $P_M$-almost all 
 $\beta \in B$, the measure $P_C(\cdot|\beta)$ can not be computed from~$\beta$.
  
 For each $\beta \in B$, the measure $P_C(\cdot|\beta)$ 
 equals the uniform measure with support $[\alpha,1]$.
 Let $R$ be this measure.
 The function $x \mapsto R([x])$ computes~$\alpha$, hence $R$ is not computable. 
 Lemma~\ref{lem:probComp} implies that 
 the set of $\beta$ that compute $R$ has $P_M$-measure zero. Hence, at most 
 a $P_M$-measure zero of $\beta \in B$ do not satisfy the conditions of the corollary. 
\end{proof}
Unfortunately, for any $[x] \subseteq [0,\alpha]$, the function $P_C([x]|\cdot)$ is nowhere continuous. 
It is only continuous in the set of points that are not binary rational, and this set 
is not negligible (it has $P_M$-measure~$\alpha$). 
Therefore, we present another example of such a measure for which the conditional measure 
is continuous, even for all~$\beta$.
\begin{theorem}\label{th:continuousNoncomputable}
  There exists a computable measure $P$ on $\mathbb{N} \times 2^\mathbb{N}$ such that:
  \begin{itemize}
    \item 
      for each $S \subseteq \mathbb{N}$, the function $P_C(S|\cdot)$ is defined and continuous on $2^{\mathbb N}$,
    \item 
      the set of $\beta$ for which 
      $P_C(\cdot|\beta)$ is not computable relative to $\beta$, has $P_M$-measure one.
  \end{itemize}
\end{theorem}

\begin{proof}
 Let $A$ be a computably enumerable set that is not computable (for example the Halting problem). 
 Fix an algorithm that enumerates the elements of $A$, and for each $n \in A$ let $t_n$ 
 be the time at which this algorithm enumerates~$n$.
 The idea of the construction of $P$ is the same as in~\cite{noncomputableConditionalMeasure}:
 if $n \not\in A$, then the measure $P(\{n\} \times \cdot)$ is uniformly distributed over $2^{\mathbb N}$.
 Otherwise, the measure is non-uniform, but only at a very small scale, i.e., for $|x| \le t_n$, 
 the values of $P(\{n\} \times [x])$ do not depend on whether $x \in A$ or not, and only for 
 $|x| > t_n$ the values are different.
 In this way, we guarantee that $P$ is computable:
 if $|x| > t_n$, a program that computes $P(\{n\} \times [x])$ on input $(n,x)$ 
 can discover whether $n \in A$ and compute the different value. 
 Because the conditional measure is defined in the limit, 
 $P_C(\cdot|\beta)$ depends on this small scale structure, 
 and therefore, the conditional measure can encode non-computable information.
 
 \smallskip
 To define $P$, we use the functions $f_0$ and $f_1$ which are defined graphically in the
 figure below. More precisely, $f_0$ is the unique piecewise linear function whose graph contains the points 
 $(0,2)$, $(1/4,0)$, $(2/4,0)$, $(3/4,2)$ and $(1,2)$. $f_1$ is defined by the relation 
 $f_1(r) = f_0( (r - 1/4) \bmod 1)$ for $r \in [0,1]$.
 Note that the average of $f_i$ over $2^\mathbb{N}$ is $1$ for $i = 0,1$. 
 For $\beta \in 2^{\mathbb N}$, let $\beta_{^{\,t}}$ be the $t$th bit of $\beta$. Note that
 \[
 \beta \mapsto f_i\left(\beta_{^{\,t+1}}\beta_{^{\,t+2}}\dots\right)
 \]
 is the function obtained by repeating $f_i$ with period $2^{-t}$.
 These functions are all continuous and have average $1$.
 \begin{figure}[h]
   \centering
 \begin{tikzpicture}
   \draw[thin, gray, dashed] (0,0) grid (4,2);
   \draw (0,2.5) -- (0,0) --  (4,0) -- (4,2.5);
   \node[anchor=north] at (2,-.2) {$2^\mathbb{N}$};
   \node[anchor=east] at (0,2) {$2$};
   \node[anchor=east] at (0,1) {$1$};

   \draw[dashed,thick] (0,2) -- (1,2) node[anchor=south] {$f_1$} -- (2,0) -- (3,0) -- (4,2);
   \draw[thick] (0,2) -- (1,0) -- (2,0) -- (3,2) node[anchor=south] {$f_0$} -- (4,2);
 \end{tikzpicture}
 \quad
 \newcommand{\drawMotive}[3]{
   \fill[path fading=east] #1 rectangle +($(#2/4,#3)$);
   \fill[white]            ($#1+(#2/4,0)$) rectangle +($(#2/4,#3)$);
   \fill[path fading=west] ($#1+(2*#2/4,0)$) rectangle +($(#2/4,#3)$);
   \draw                   ($#1+(3*#2/4,0)$) -- +(0,#3);
   \fill                   ($#1+(3*#2/4,0)$) rectangle +($(#2/4,#3)$);
   \draw                   ($#1+(#2,0)$) -- +(0,#3);
 }
 \begin{tikzpicture}[yscale=0.8]
   \draw (0,0) rectangle (4,4);
   \node[anchor=north] at (2,-.2) {$2^\mathbb{N}$};
   \node[anchor=east] at (0,3) {$n=1$};
   \node[anchor=east] at (0,1.5) {$n=2$};
   \node[anchor=east] at (0,0.75) {$n=3$};
   \node[anchor=east] at (0,0.375) {$n=4$};

   \fill[black!30] (0,2) rectangle (4,4);

   \foreach \x in {0,.25,...,3.75}
     {\drawMotive{(\x,1)}{.25}{1}}

   \fill[black!30] (0,.5) rectangle (4,1);

   \foreach \x in {0,.1,...,4}
     {\drawMotive{(\x,.25)}{.1}{.25}}

   \foreach \x in {0,.5,...,3.5}
     {\drawMotive{(\x,0.125)}{.5}{0.125}}

   \fill[black!30] (0,.06) rectangle (4,.125);
 \end{tikzpicture}
 \caption{Left: measures $f_0$ and $f_1$ over $2^{\mathbb N}$. Right: measure $P$ over $\mathbb N \times 2^{\mathbb N}$.}
 \label{fig:densities}
 \end{figure}
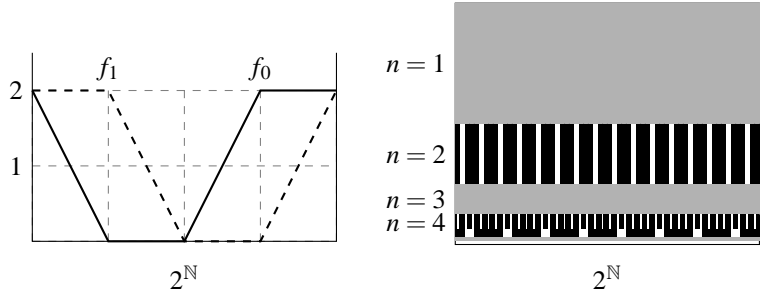
  
 \noindent
 Let us first define $P$ using the following density, see figure~\ref{fig:densities}:
  \[
   f(n, \beta) = 
    \begin{cases}
      2^{-n} f_0(\beta_{^{\,t_n+1}}\beta_{^{\,t_n+2}}\dots)  & \text{if $n \in A$}\\
      2^{-n} & \text{otherwise.}
   \end{cases}
  \]
  Thus, $P(\{n\} \times [x]) = \int_{[x]} f(n,\beta) \text{d}\beta$.

  Let $P(n|\beta)$ be short for $P_C(\{n\}|\beta)$. We now show that this function is continuous in $\beta$.
  The marginal density $f_M = \sum_{i \in \mathbb{N}} f(i, \cdot)$ is continuous, 
  because it is a uniformly convergent sum of continuous functions. 
  Also, $f_M$ is bounded from below by a positive constant 
  (if $m \not\in A$, then $f_M \ge 2^{-m}$).  
  Hence, the conditional measure is continuous on singleton sets:
  \[
    P(n|\beta) = \frac{f(n,\beta)}{\sum_{i \in \mathbb N}f(i,\beta)}.
  \]
  For $S \subseteq \mathbb N$, $P(S|\beta)$ is a uniformly convergent sum of continuous functions, 
  and hence also continuous.

  \medskip
  By Lemma~\ref{lem:probComp}, it remains to show for each $\beta$ that $P(\cdot|\beta)$ computes $A$,
  (i.e., $A$ is computed by a machine that has oracle access to approximations of $P(n|\beta)$ of any precision).
  For each fixed $\beta$, the values of $P(n|\beta)$ for all $n \not\in A$ are the same.
  Unfortunately, there can be many $n \in A$ for which $P(n|\beta)$ is close to this value.
  Hence, $A$ might not be computable from $P(\cdot|\beta)$.

  \smallskip
  We now adapt the construction. The new measure $\hat P$ encodes 
  membership of $n$ in $A$ using two values of the conditional measure:
  $\hat P(2n|\beta)$ and $\hat P(2n+1|\beta)$.  Note
  that for each $\beta \in 2^{\mathbb{N}}$ at least one of the values
  $f_0(\beta), f_1(\beta)$ 
  is either $0$ or $2$.  Hence, for $b \in \{0,1\}$, we define $\hat P$ using
  \[
   f(2n + b, \beta) = 
    \begin{cases}
      2^{-2n-b} f_b(\beta_{^{t_n+1}}\beta_{^{t_n+2}}\dots)  & \text{if $n \in A$,}\\
      2^{-2n-b} & \text{otherwise.}
   \end{cases}
  \]
  For the same reasons as before, $\hat P(n|\cdot)$ is continuous.
  Fix an $m \not\in A$ and note that $\hat P(2m|\beta)>0$.
  If $n \in A$ at least one of the values 
  \[
  \frac{\hat P(2n|\beta)2^{2n}}{\hat P(2m|\beta)2^{2m}}, \frac{\hat P(2n+1|\beta)2^{2n+1}}{\hat P(2m|\beta)2^{2m}}
  \]
  equals $0$ or $2$; otherwise, both values equal $1$.
  Hence, for each $\beta$ the measure $\hat P(\cdot|\beta)$ computes $A$. 
  (Because $\inf_\beta \hat P(2m|\beta) > 0$, the reduction can even be made uniformly 
  in $\hat P(\cdot|\beta)$ for all~$\beta$.)
\end{proof}

\section*{Appendix: Two definitions of conditional measure coincide}

In probability theory, conditional measures are defined implicitly 
using the Radon-Nikodym theorem. Any measure that satisfies the conditions 
of this theorem can be used as a conditional measure. The following lemma 
states that such measures are almost everywhere equal to the conditional 
measure $P_C$ defined above.

\begin{lemma}[Folklore]\label{lem:RadonNikodym}
  Let $2^*$ be the set of strings. 
  For every measure $P$ on $2^{\mathbb N} \times 2^{\mathbb N}$ 
  and for every function $f: 2^* \times 2^{\mathbb N}$ such that for all $x$ and $y$
  \[
  P([x],[y]) = \int_{[y]} f(x,\beta) P_M(\text{d}\beta),
  \]
  we have that $f(\cdot,\beta) = P_C([\cdot]|\beta)$ 
  for all $\beta$ in a set of $P_M$-measure one.
\end{lemma}

In the proof we use the Lebesgue differentiation theorem for Cantor space. 
The proof of this version follows from the proof for Real numbers.
\begin{theorem}[Lebesgue differentiation theorem for Cantor space]\label{lem:LebesgueDifferentiationTheorem}
  Let $Q$ be a measure on $2^{\mathbb N}$. For every $Q$-integrable function $g:2^{\mathbb N} \rar \mathbb R$ we have that
\[
\lim_n \frac{\int_{[\beta_1,\dots,\beta_n]} g(\gamma) Q(\text{d}\gamma)}{Q([\beta_1,\dots,\beta_n])} = g(\beta)
   \]
 for $Q$-almost all $\beta$.
\end{theorem}

\begin{proof}[of  Lemma~\ref{lem:RadonNikodym}]
 For a fixed $x$, apply the Lebesgue differentiation theorem with $g(\cdot) = f(x,\cdot)$ and $Q = P_M$. 
 By assumption on $f$, the nominator simplifies to $P([x],[\beta_1\dots \beta_n])$.
 It follows that $f(x,\beta)$ differs from $P_C([x]|\beta)$ in at most a set of $\beta$ with $P_M$-measure zero.
 Because there are countably many strings $x$, it follows that $f(\cdot,\beta)$ and $P_C([\cdot]|\beta)$ 
 differ in at most a set of $P_M$-measure zero.
\end{proof}

\bibliography{kolmogorov}   

\end{document}